\documentclass[a4paper,10pt]{article}
\date{}
\usepackage{authblk}
\usepackage[english]{babel}  
\usepackage[utf8]{inputenc}
\usepackage{hyperref}
\usepackage{graphicx}
\RequirePackage{latexsym} \RequirePackage{amsthm}
\RequirePackage{amsmath}

\usepackage{enumerate}
\RequirePackage{amssymb} \RequirePackage{makeidx}
\usepackage{dsfont}
\usepackage{mathrsfs}
\usepackage{mathtools}
\newcommand{\Lim}[1]{\raisebox{0.5ex}{\scalebox{0.8}{$\displaystyle \lim_{#1}\;$}}}

\newtheorem{theorem}{Theorem}[section]
\newtheorem{lemma}[theorem]{Lemma}

\newtheorem{definition}[theorem]{Definition}

\begin{document}

\title{Universal partial sums of Taylor series as functions of the centre of expansion}

\date{}

\author{Christoforos Panagiotis \thanks{Supported by the European Research Council (ERC) under the European Union's Horizon 2020 research and innovation programme (grant agreement No 639046).}}

\affil{{Mathematics Institute\\University of Warwick\\CV4 7AL, UK\\ 
Email \href{mailto:C.Panagiotis@warwick.ac.uk}{C.Panagiotis@warwick.ac.uk} }}

\maketitle

\begin{abstract}
V. Nestoridis conjectured that if $\Omega$ is a simply connected subset of $\mathbb{C}$ that does not 
contain $0$ and $S(\Omega)$ is the set of all functions $f\in \mathcal{H}(\Omega)$ with the property that the set 
$\left\{T_N(f)(z)\coloneqq\sum_{n=0}^N\dfrac{f^{(n)}(z)}{n!} (-z)^n : N = 0,1,2,\dots \right\}$ is dense in $\mathcal{H}(\Omega)$, then $S(\Omega)$ is a dense 
$G_\delta$ set in $\mathcal{H}(\Omega)$. We answer the conjecture in the affirmative in the special case where
$\Omega$ is an open disc $D(z_0,r)$ that does not contain $0$.
\end{abstract}
\vspace{5pt}
\begin{flushleft} AMS Classification numbers: primary 30K99, secondary 30K05
  \end{flushleft}
\vspace{5pt}
\begin{flushleft}
  Key words and phrases: Taylor Expansion, Partial Sums, Universal Taylor Series, Ostrowski-gaps, Baire's Category Theorem, Generic Property.
\end{flushleft}

\section{Introduction}
Let $\Omega$ be a simply connected domain in the complex plane and consider the space $\mathcal{H}(\Omega)$ of all 
the functions that are holomorphic on $\Omega$ with the topology of uniform convergence on compacta. If $f$ 
belongs to $\mathcal{H}(\Omega)$ and $\zeta$ is a point in $\Omega$, then $S_N(f,\zeta)(z)$ denotes the $N$-th 
partial sum of the Taylor expansion of $f$ with centre $\zeta$ at $z$, i.e. 
\begin{align*}
S_N(f,\zeta)(z)=\sum_{n=0}^N\dfrac{f^{(n)}(\zeta)}{n!} (z-\zeta)^n.
\end{align*} 
The last expression converges to $f(z)$ as $N\rightarrow \infty$ whenever $z\in D(\zeta,r)$ for some sufficiently small $r>0$, such that $D(\zeta,r)\subset \Omega$. This expression makes also sense when $\zeta\in\Omega$ and $z\in \Omega^{\mathsf{c}}$, but now $S_N(f,\zeta)(z)$ may or may not have a limit. A natural question is what kind of behaviour can the partial sums exhibit in the later case.

Overconvergence is the phenomenon of convergence of a certain subsequence of partial sums on a domain that is larger than the domain of convergence of the series. It turns out that there exist power series which satisfy this definition to its extreme in the following sense.

\begin{theorem}
Let $0\leq r<\infty$. There exists a power series $\sum_{n=0}^\infty a_n z^n$ of
radius of convergence $r$ such that for every compact set $K$ in $\{z : |z| > r\}$ with
connected complement and every function $h$ that is continuous on $K$ and holomorphic on the interior of $K$ there exists an increasing sequence $\{\lambda_n\}\in\{0, 1, 2, \dots\}$ such that
\begin{align*}
\sum_{k=0}^{\lambda_n} a_k z^k \rightarrow h(z)
\end{align*}
as $k\rightarrow \infty$
uniformly on $K$.
\end{theorem}

The result was proved in the case $r=0$ by Seleznev \cite{11} and
in the general case by Luh \cite{2} and Chui and Parnes \cite{1}. Several extensions and
generalisations have been considered \cite{3},\cite{4},\cite{6},\cite{7},\cite{8},\cite{9}.
Among them, of particular interest to us is the following. 

\begin{definition}\label{universal}
The set $U(\Omega)$ is the set of all functions $f\in \mathcal{H}(\Omega)$ with the property
that, for every compact set $K \subset \mathbb{C}$, 
$K \cap \Omega=\emptyset$, with $K^{\mathsf{c}}$
connected, and every function $h$ which is continuous on $K$ and holomorphic in the interior of $K$, 
there exists a sequence $\{\lambda_n\}\in\{0, 1, 2, \dots\}$ such that, for every compact set 
$L\subset \Omega$,
$$\sup_{\zeta\in L}\sup_{z\in K}|S_{\lambda_n}(f, \zeta)(z) - h(z)| \rightarrow 0,
\hspace{10pt} n \rightarrow \infty.$$
\end{definition}
\noindent
In \cite{9} it was proven that the 
set $U(\Omega)$ is a dense and $G_\delta$ subset of $\mathcal{H}(\Omega)$.

Let us now consider a simply connected domain $\Omega$ that does not contain $0$. Contrary to the functions of $U(\Omega)$,  fix $z=0$ and let 
$\zeta$ vary in $\Omega$. Define $$T_n(f)(z)=\sum_{n=0}^N\dfrac{f^{(n)}(z)}{n!} (-z)^n.$$ In \cite{12} V. Nestoridis conjectured 
that if $S(\Omega)$ is the of all functions $f\in \mathcal{H}(\Omega)$ with the property that the set 
$\{T_N(f) : N = 0,1,2,\dots \}$ is dense in $\mathcal{H}(\Omega)$, then $S(\Omega)$ is a dense $G_\delta$ set in $\mathcal{H}(\Omega)$. As noted 
in \cite{12}, if $f$ belongs 
to $U(\Omega)$ and we select $K=\{0\}$ in Definition \ref{universal}, then there exists a function in $\mathcal{H}(\Omega)$ with the property that the set 
$\{T_N(f) : N = 0,1,2,\dots \}$ approximates the constant functions in $\Omega$. The last statement is clearly evidence for the truth of the conjecture.

In the present paper, we answer the conjecture of V. Nestoridis in the affirmative in the special case where $\Omega$ is 
an open disc $D(z_0,r)$ that does not contain $0$. We prove that the set of holomorphic function $f$ 
defined on $D(z_0,r)$ such that $\{T_n(f)\}$ 
approximates every function of the form $P\circ \log$, where $P$ is a polynomial with coefficients in 
$\mathbb{Q}+i\mathbb{Q}$ and $\log$ is a logarithm defined on $D(z_0,r)$, is a dense $G_\delta$ set in 
$\mathcal{H}(D(z_0,r))$. This set coincides with the original one of the conjecture. In order to do so, we construct a function $g\in \mathcal{H}(D(z_0,r))$ having Ostrowski-gaps, such that the set $\{T_N(f) : N = 0,1,2,\dots \}$ 
approximates a function of the form $P\circ \log$ for a fixed polynomial $P$ and
use Baire's category theorem to prove the aforementioned special case of the conjecture. 

\section{Main result}

Consider a complex number $z_0$ and an open disc $D(z_0,r)=\{z\in\mathbb{C}, |z-z_0|<r\}$, $r>0$ such that $0\not\in D(z_0,r)$. We will denote $T^{'}_n(f)(z)$ the derivative of $T_n(f)(z)$ with respect to the centre of expansion $z$. Let us start by making some useful observations, which lay the foundations for what we consider as our main result. Their proofs are simple exercises and they are omitted.

\begin{lemma}\label{derivative}
Let $\Omega$ be an open subset of $\mathbb{C}$ and $f$ a holomorphic function on $\Omega$. Then 
$$T^{'}_n(f)(z)=\dfrac{f^{(n+1)}(z)}{n!}(-z)^n.$$
\end{lemma}

\begin{lemma}\label{zf(z)}
Let $\Omega$ be an open subset of $\mathbb{C}$, $f$ a holomorphic function on 
$\Omega$ and $g(z)=z f(z)$. 
Then
$$T_n(g)(z)=z\dfrac{f^{(n)}(z)}{n!}(-z)^n.$$
\end{lemma}

A key part in our proof of the main result will be the construction of functions that satisfy certain properties. First however, let us give a definition.

\begin{definition}\label{Ostrowski-gaps}
Let $\sum_{n=0}^\infty a_n(z-z_0)^n$ be a power series with radius of
convergence $R\in (0,\infty)$, and let $(p_n)_{n\in\mathbb{N}}$ and $(q_n)_{n\in
\mathbb{N}}$ be sequences of natural numbers with $1\leq p_1 < q_1\leq p_2 < q_2 \leq \dots$. We say that 
the power series or the sequence of partial sums has Ostrowski-gaps
$$(p_n, q_n)_{n\in \mathbb{N}}$$
if $q_n/p_n \rightarrow \infty$ for $n\rightarrow \infty$, and if for $I =
\cup_{n=1}^\infty [p_n + 1, q_n]$ we have $\Lim{k \rightarrow \infty,k\in I} |a_k|^{1/k} = 0.$
\end{definition}

For our purposes we need two sequences defined slightly differently from the sequences $(p_n)_{n\in\mathbb{N}}$ and $(q_n)_{n\in
\mathbb{N}}$  in Definition \ref{Ostrowski-gaps}. More precisely, we need two sequences of natural numbers $(p_n)_{n\in\mathbb{N}}$ and $(q_n)_{n\in
\mathbb{N}}$ with $q_n/p_n \rightarrow \infty$ for $n\rightarrow \infty$, $1\leq p_1 < q_1+1< p_2 < q_2+1 
< \dots$ and $p_{n+1}-q_n \rightarrow \infty$. We fix two such sequences.

Suppose now that we are given a sequence $(c_n)_{n\in\mathbb{N}}$ of complex numbers with the property
\begin{equation}\label{limsup}
\limsup_{n\rightarrow\infty} |c_n|^{1/n} \leq \dfrac{|z_0|}{r}.
\end{equation}
Note that $\dfrac{|z_0|}{r}\geq 1$, since $0\not\in D(z_0,r)$. Hence, there are plenty of such sequences $c_n$, for example the constant sequences.
Given such a sequence, we will construct a power series with Ostrowski-gaps 
$(p_n,q_n)_{n\in\mathbb{N}}$ and radius of convergence     
$ R\geq r$, such that 
\begin{equation}\label{construction}
T_{p_n}(g)-c_{p_n}\rightarrow 0
\end{equation}
uniformly on the compact subsets of $D(z_0,R)$.
To this end, let us define the Taylor series coefficients
\[a_k= \begin{cases} 
      \dfrac{c_{p_n}}{(-z_0)^{p_n}} & k=p_n, n\in\mathbb{N} \\
      \dfrac{-c_{p_n}}{(-z_0)^{q_n+1}} & k=q_n+1, n\in\mathbb{N} \\
      0 & otherwise
   \end{cases}
\]
and 
\begin{equation*}
g(z)=\sum_{k=1}^\infty a_k (z-z_0)^k.
\end{equation*}
By the definition of $a_k$ and \eqref{limsup} we have that 
\begin{equation*}
\limsup_{k\rightarrow\infty} |a_k|^{1/k}\leq \dfrac{1}{r}
\end{equation*}
and thus the radius of convergence, $R$, of the power series $\sum_{k=1}^\infty a_k (z-z_0)^k$  is
\begin{equation*}
R =\dfrac{1}{\limsup_{k\rightarrow\infty} |a_k|^{1/k}}\geq r.
\end{equation*}
Also, notice that $T_{p_n}(g)(z_0)=c_{p_n}$, $T_{k}(g)(z_0)=0$ for $k=q_n+1,\dots, p_{n+1}-1$ 
and that $g$ has Ostrowski-gaps $(p_n,q_n)_{n\in\mathbb{N}}$.

In \cite{7} (Lemma 9.2) the authors prove the following result, following the proof of Theorem 1 in \cite{5}.

\begin{lemma}\label{Melas Nestoridis}
Let $f(z)=\sum_{\nu=0}^\infty a_\nu(z-z_0)^\nu$ be the Taylor development of
a holomorphic function on an open set $\Omega$ around the point $z_0$. Suppose
that this series has Ostrowski gaps $(p_n ,q_n)_{n\in\mathbb{N}}$.
Then the difference $S_{p_n}( f,\zeta)(w)-S_{p_n}(f,z_0)(w)$ converges to $0$ (as $n\rightarrow\infty$)
uniformly on the compact subsets of $\Omega\times \mathbb{C}$.
\end{lemma}

In particular, under the assumptions of Lemma \ref{Melas Nestoridis}  
$$S_{p_n}(f,\zeta)(z)-S_{p_n}(f,z_0)(z)$$ converges to $0$
uniformly on the compact subsets of $D(z_0,r)\times\{0\}$.
By definition of $T_{p_n}(g)$, we have
\begin{equation*}
T_{p_n}(g)(z)-c_{p_n}\rightarrow 0
\end{equation*}
uniformly on the compact subsets of $D(z_0,R)$ and especially on 
the compact subsets of $D(z_0,r)$.

The following result will be useful to us. If $g$ is a holomorhic function defined on a neighbourhood of $z_0$ and its Taylor expansion has radius of convergence $0<R<\infty$ at $z_0$, then
\begin{equation}\label{equation}
\limsup_{n\rightarrow\infty}(\max_{|z-z_0| 
\leq \rho R}|S_{n}(g,z_0)(z)|^{1/n})=\rho
\end{equation}
for every $\rho\geq 1$. This follows easily from the Cauchy estimates. Indeed,
for any arbitrary $\varepsilon>0$ we have for the Taylor coefficients, $a_n$, of
$g$ that $|a_n|\leq (1/R+\varepsilon)^n$ for every large enough natural number 
$n$. Hence, there is a constant $C>0$, such that 
$$|S_n(g,z_0)(z)|\leq \sum_{k=0}^n C(1/R+\varepsilon)^k(\rho R)^k=C\dfrac{
(\rho+\varepsilon\rho R)^{n+1}-1}{\rho+\varepsilon\rho R-1}$$
for every natural number $n$ and every $z$ such that $|z-z_0|\leq \rho R$.
The left hand side of \eqref{equation} is at most $\rho+\varepsilon \rho R$. Since, 
$\varepsilon>0$ was arbitrary it follows that the left hand side of \eqref{equation} is at most $\rho$. On the other hand,
if the left hand side were less than $\rho-\varepsilon$, then from the Cauchy estimates $$|a_n|=|S_n(g,z_0)^{(n)}(z_0)|\leq 1/(\rho R)^n \max_{|z-z_0| 
\leq \rho R}|S_{n}(g,z_0)(z)|,$$
we get that $$|a_n|\leq (1/R-\varepsilon/2R)^n$$ for every large enough 
$n$. However, this contradicts our assumption that the radius of convergence is 
exactly $R$.

In what follows we use Lemma \ref{derivative}, Lemma \ref{zf(z)}, the above construction and the aforementioned result to prove the existence of functions $f$ such that $\{T_n(f)\}$ approximates certain functions. 

\begin{theorem}
Let $P$ be a polynomial. Then there exists a function $Q_P\in \mathcal{H}(D(z_0,r))$ such that 
$$T_{p_n}(Q_P)(z)\rightarrow (P\circ \log)(z)$$
uniformly on the compact subsets of $D(z_0,r)$.
\end{theorem}

\begin{proof}
Let $d$ be the degree of $P$. Without loss of generality we can assume that $d\geq 1$. To prove the statement of the Theorem, we will at first approximate functions of the form $Q\circ \log$, where $Q$ is a polynomial of lower degree. 

We will start from the constant polynomial $Q\equiv 1$. Choosing $(c_n)_{n\in\mathbb{N}}$ to be the constant sequence equal to $1$, we can construct a 
function $g$ such that $T_{p_n}(g)\rightarrow 1 $ uniformly on the compact subsets of $D(z_0,r)$. Let us define $f(z)=g(z)/z$, which is a holomorphic function on $D(z_0,r)$, because $0\not\in D(z_0,r)$. Consider an antiderivative, $F$, of $f$.
Combining Lemma \ref{derivative} and Lemma \ref{zf(z)}, we have
\begin{equation}\label{combination}
T_n(g)(z)=z\dfrac{f^{(n)}(z)}{n!}(-z)^n=z\dfrac{F^{(n+1)}(z)}{n!}(-z)^n=zT^{'}_n(F)(z).
\end{equation}
Since $T_{p_n}(g)\rightarrow 1 $ uniformly on the compact subsets of $D(z_0,r)$, we conclude that
\begin{equation}\label{1/z}
T^{'}_{p_n}(F)(z)\rightarrow 1/z
\end{equation}
uniformly on the compact subsets of $D(z_0,r)$.
It follows by integration that
\begin{equation*}
T_{p_n}(F)(z)-T_{p_n}(F)(z_0)-\log(z)+\log(z_0)\rightarrow 0
\end{equation*}
uniformly on the compact subsets of $D(z_0,r)$.
Our goal is to find a function $\phi\in \mathcal{H}(D(z_0,r)$, such that
\begin{equation*}
T_{p_n}(\phi)(z)-T_{p_n}(F)(z_0)+\log(z_0)\rightarrow 0
\end{equation*}
uniformly on the compact subsets of $D(z_0,r)$. First, we need to prove that
\begin{equation*}
\limsup_{n\rightarrow\infty} |T_n(F)(z_0)|^{1/n}\leq|z_0|/r.
\end{equation*}
Since $p_{n+1}-q_{n}\rightarrow\infty$, there is $k_0\in\mathbb{N}$ such that 
$p_{n+1}-q_{n}> 3$ for every $n\geq k_0$.
Define the function $h$ to be the power series $\sum_{k=0}^\infty a_k(h)(z-z_0)^k$,
where
\[a_k(h)= \begin{cases} 
      \dfrac{1}{r^{q_n+2}} & k=q_n+2, n\geq k_0 \\
      \dfrac{1}{z_0r^{q_n+2}} & k=q_n+3, n\geq k_0 \\
      0 & otherwise.
   \end{cases}
\]
Then $h$ is a holomorphic function on $D(z_0,r)$ and $T_{p_n}(h)(z_0)=0$. Also, the radius of convergence of the function $F+h$ is at least $r$, since it is holomorphic on $D(z_0,r)$.

Moreover $T_{k}(g)(z_0)=0$ for 
$k=q_n+1,\dots,p_{n+1}-1$, which combined with \eqref{combination} implies for the Taylor coefficients of $F$ at 
$z_0$ that
\begin{equation*}
a_k(F)=\dfrac{F^{(k)}(z_0)}{k!}=0, k=q_n+2,\dots, p_{n+1}.
\end{equation*}
Thus, for the Taylor coefficients of $F+h$ at $z_0$ we have $a_{q_n+2}(F+h)=\dfrac{1}{r^{q_n+2}}$ for 
$n\geq k_0$, which implies that the radius of convergence of the power series of $F+h$ is exactly $r$.
Now, we notice that $\rho=\dfrac{|z_0|}{r}\geq 1$ and hence
\begin{equation*}
\limsup_{n\rightarrow\infty} (\max_{|z-z_0| 
\leq \rho r}|S_{n}(F+h,z_0)(z)|^{1/n})=\rho.
\end{equation*}
Obviously, $0\in \overline{D(z_0,\rho r)}$ and thus
\begin{equation}\label{pastitsio}
\limsup_{n\rightarrow\infty} |T_n(F)(z_0)|^{1/n}=\limsup_{n\rightarrow\infty} |T_n(F+h)(z_0)|^{1/n}\leq\rho=|z_0|/r.
\end{equation}
Also, since 
\begin{equation*}
\limsup_{n\rightarrow\infty} |-\log(z_0)|^{1/n}=1\leq|z_0|/r
\end{equation*} 
we can construct a holomorphic 
function $\phi$ defined on $D(z_0,r)$, such that
\begin{equation*}
T_{p_n}(\phi)(z)-T_{p_n}(F)(z_0)+\log(z_0)\rightarrow 0
\end{equation*}
uniformly on the compact subsets of $D(z_0,r)$. Therefore,
\begin{equation}\label{log}
T_{p_n}(F+\phi)(z)\rightarrow \log(z).
\end{equation}
If $d=1$, then combining \eqref{log} with the fact that $T_{p_n}(g)\rightarrow 1$ and the linearity of $T_k$ we can stop our proof here and conclude. Otherwise, 
our next step is going to be the approximation of $\dfrac{\log^2(z)}{2}$. This can be done using the above arguments. However, we need to define a function 
$\psi$, such that $F+\phi+\psi$ has the additional property that
\begin{equation*}
T_k(F+\phi+\psi)(z_0)=0
\end{equation*}
for $k=q_n+3,\dots, p_{n+1}-2$, for every large enough $n$. 
In order to do so, consider at first a natural 
number $k_1$ such that $p_{n+1}-q_n>4$ for every $n\geq k_1$. Now, define the power 
series $\psi(z)=\sum_{k=0}^\infty a_k(\psi)(z-z_0)^k$,
where
\[a_k(\psi)= \begin{cases} 
      \dfrac{-T_{q_n+3}(F+\phi)(z_0)}{(-z_0)^{q_n+3}} & k=q_n+3, n\geq k_1 \\
      \dfrac{T_{q_n+3}(F+\phi)(z_0)}{(-z_0)^{p_{n+1}-1}} & k=p_{n+1}-1, n\geq k_1 \\
      0 & otherwise.
   \end{cases}
\]
It is easy to see that $\psi$ is a holomorphic function on $D(z_0,r)$. Indeed, notice that the radius of convergence of the power series of $F+h+\phi$ is exactly $r$. Thus, similarly to \eqref{pastitsio}
\begin{equation*}
\limsup_{n\rightarrow\infty} |T_n(F+\phi)(z_0)|^{1/n}\leq\rho=|z_0|/r.
\end{equation*}
and hence the radius of convergence of $\psi$ is greater than or equal to $r$.

Also, 
\begin{equation*}
T_{p_n}(\psi)(z_0)=0
\end{equation*}
and, using again Lemma \ref{Melas Nestoridis}, we have that
\begin{equation*}
T_{p_n}(\psi)(z)\rightarrow 0
\end{equation*}
uniformly on the compact subsets of $D(z_0,r)$.
Thus, 
\begin{equation*}
T_{p_n}(F+\phi+\psi)(z)\rightarrow \log(z)
\end{equation*}
and 
\begin{equation*}
T_k(F+\phi+\psi)(z_0)=0
\end{equation*}
for $k=q_n+3,\dots, p_{n+1}-2$, where $n\geq k_1$.
Using these two properties, we can proceed by repeating the above arguments for 
$F+\phi+\psi$ in place of $g$, in order to approximate the antiderivative of $\log(z)/z$, which is 
$\log^2(z)/2$. 

For the general case, since the antiderivative of $\log^{k-1}(z)/z$ is $\log^k(z)/k$, we can proceed 
inductively to find a function $Q_k\in \mathcal{H}(D(z_0,r))$ such that $T_{p_n}(Q_k)$ approximates $\log^k(z)/k!$. From the 
linearity of the operators $T_n$, there is a function $Q_P\in \mathcal{H}(D(z_0,r))$ such that 
\begin{equation*}
T_{p_n}(Q_P)(z)\rightarrow (P\circ \log)(z)
\end{equation*}
uniformly on the compact subsets of $D(z_0,r)$.
\end{proof}

We are now ready to prove our main result.

\begin{theorem}
The set $S(D(z_0),r))$ of all functions $f\in \mathcal{H}(D(z_0),r))$ with the property that the set $\{T_N(f) : N = 0, 1, 
2, \dots \}$ is dense in $\mathcal{H}(D(z_0),r))$ is a dense $G_\delta$ set in $\mathcal{H}(D(z_0),r))$
\end{theorem}

\begin{proof}
We will start by providing a convenient set expression for $S(D(z_0),r))$. Let $D_m$, $m\geq 1$ be the closed disc $\overline{D(z_0,r(1-1/m))}$ and $(P_j)_{n\in\mathbb{N}}$ an enumeration
of the polynomials with coefficients in $\mathbb{Q}+i\mathbb{Q}$. Define
$$A(m,j,s)=\bigcup_{n=0}^\infty \{f\in D(z_0,r): \sup_{z\in D_m} |T_n(f)(z)-(P_j\circ \log)(z)|<1/s\}.$$
We claim that the set
$$A=\bigcap_{m=1}^\infty\bigcap_{j=1}^\infty\bigcap_{s=1}^\infty A(m,j,s)$$ 
coincides with the desired one.

Indeed, $S(D(z_0,r))$ is clearly a subset of $A$. It remains to prove that $A$ is a subset of $S(D(z_0,r))$. Let $g\in A$, $f\in \mathcal{H}(D(z_0,r))$ and $K$ be a compact subset of $D(z_0,r)$. It is easy to see that there is a natural 
number $N$ such that $K\subset D_N$. Since $D(z_0,r)$ is a simply connected open set and $\log(z)$ is a 
conformal mapping, it is well known that $\Omega=\log(D(z_0,r))$ is a simply connected open set. Also, the 
function $f\circ \exp$, where $\exp$ is the exponential function, is a holomorphic function on $\Omega$. By 
Runge's approximation theorem \cite{10}, there is a sequence of polynomials $(Q_n)_{n\in\mathbb{N}}$ with 
coefficients in $\mathbb{Q}+i\mathbb{Q}$ that converges to $f\circ \exp$ uniformly on the compact 
subsets of $\Omega$. Then $(Q_n\circ \log)_{n\in\mathbb{N}}$ converges to $f$ uniformly on the 
compact subsets of $D(z_0,r)$. Now, pick a sequence of natural numbers 
$(k_n)_{n\in\mathbb{N}}$, such that 
\begin{equation*}
\sup_{z\in D_n}|T_{k_n}(g)(z)-(Q_n\circ \log)(z)|\leq 1/n.
\end{equation*}
Since $(Q_n\circ \log)_{n\in\mathbb{N}}$ converges to $f$ uniformly on $D_N$ 
and $K\subset D_N$, we conclude that $T_{k_n}(g)(z)$ converges to $f$ uniformly on $K$. The compact set $K\subset D(z_0,r)$ was arbitrary. Thus, $T_{k_n}(g)(z)$ converges to $f$ uniformly on the compact subsets of $D(z_0,r)$.
It follows that $A$ is a subset of $S(D(z_0,r))$.

We will now prove that $A(m,j,s)$ is dense. To this end, let $g\in \mathcal{H}(D(z_0,r))$, $K$ be a compact subset of $D(z_0,r)$ and $\varepsilon>0$. Consider also a function $h$ such 
that $T_{p_n}(h)(z)\rightarrow (P_j\circ \log)(z)$ uniformly on the compact subsets of $D(z_0,r)$ and $N$ large enough such that $K\subset D_N$. From 
Runge's approximation theorem there exists a polynomial $P$ such that 
\begin{equation}\label{almost there}
\sup_{z\in D_N}|g(z)-h(z)-P(z)|<\varepsilon
\end{equation}
and
\begin{equation*}
|P(0)|<1/2s.
\end{equation*}
Notice that if $n$ is greater than the degree of $P$, then $T_n(P)(z)=P(0)$. If we pick now a natural 
number $n_0$ such that $p_{n_0}$ is greater than the degree of $P$ and 
\begin{equation*}
\sup_{z\in D_m}|T_{p_{n_0}}(h)(z)-(P_j\circ \log)(z)|<1/2s,
\end{equation*}
then using the triangle inequality we have
\begin{equation*}
\sup_{z\in D_m}|T_{p_{n_0}}(h+P)(z)-(P_j\circ \log)(z)|<1/s.
\end{equation*}
Thus, $h+P$ belongs to $A(m,j,s)$ which combined with the fact that $K\subset D_N$ and \eqref{almost there} implies that $A(m,j,s)$ is dense. 

Finally, $A(m,j,s)$ is an open set, since $T_n$ is a continuous operator. Baire's category theorem implies that $A$ is a dense 
$G_\delta$ subset of 
$\mathcal{H}(D(z_0,r))$. This completes the proof.
\end{proof}

\vspace{3pt}
\textbf{\textit{Acknowledgement}}---
I would like to thank V. Nestoridis, E. Archer and G. Vasdekis for taking interest in this work.

\end{document}